\documentclass[11pt]{amsart}

\usepackage{graphicx, mathabx}
\usepackage[all]{xy}
\usepackage{hyperref}

\newcommand{\dfn}[1]{{\textbf {#1}}}
\newcommand{\leg}{\ensuremath{\Lambda}}
\newcommand{\te}{\tilde{\e}}
\newcommand{\e}{\ensuremath{\eta}} 
\newcommand{\sleg}{\ensuremath{\mathcal{L}}}
\newcommand{\M}{\ensuremath{\mathcal{M}}}
\newcommand{\F}{\ensuremath{\mathcal{F}}}
\newcommand{\Z}{\ensuremath{\mathcal{Z}}}

\DeclareMathOperator{\End}{End}
\DeclareMathOperator{\Aut}{Aut}

\DeclareMathOperator{\Crit}{Crit}

\newcommand{\rr}{\ensuremath{\mathbb{R}}}
\newcommand{\zz}{\ensuremath{\mathbb{Z}}}

\newcommand{\nn}{\ensuremath{\mathbb{N}}}

\theoremstyle{plain}
\newtheorem{thm}{Theorem}[section]
\newtheorem{cor}[thm]{Corollary}
\newtheorem{lem}[thm]{Lemma}

\newtheorem{prop}[thm]{Proposition}

\theoremstyle{definition}

\theoremstyle{remark}
\newtheorem{rem}[thm]{Remark}
\newtheorem{ex}[thm]{Example}

\numberwithin{equation}{section}

\def\dfn#1{{\textbf {#1}}}

\begin{document}

\title{Families of Legendrian Submanifolds via Generating Families}

\date{\today}

\author[J. Sabloff]{Joshua M. Sabloff} \address{Haverford College,
Haverford, PA 19041} \email{jsabloff@haverford.edu} \thanks{JS was
partially supported by NSF grant DMS-0909273.}

\author[M. Sullivan]{Michael G. Sullivan} \address{University of Massachusetts, Amherst 01003} \email{sullivan@math.umass.edu} \thanks{MS was
partially supported by NSF grant DMS-1007260.}

\begin{abstract} We investigate families of Legendrian submanifolds of $1$-jet spaces by developing and applying a theory of families of generating family homologies.  This theory allows us to detect an infinite family of loops of Legendrian $n$-spheres  embedded in the standard contact $\rr^{2n+1}$ (for $n>1$) that are  contractible in the smooth, but not Legendrian, categories.
\end{abstract}

\maketitle


\section{Introduction}
\label{sec:intro}

A central motivating question in contact topology is  the search for the boundary between  flexibility (when contact objects behave like smooth objects) and rigidity (when  behavior is more restrictive).  This search tends to take the form of distinguishing or classifying contact objects up to isotopy.  Phrased in terms of the space of all contact structures on a given manifold, or the space of all Legendrians in a given contact manifold, investigating isotopy classes can be thought of as trying to understand the set of path components.  Flexibility results tend to give information about higher homotopy groups as well as $\pi_0$:   Eliashberg proved, for example, that there is a homotopy equivalence between the space of over twisted contact structures  and the set of smooth $2$-plane distributions on a $3$-manifold \cite{yasha:overtwisted}, and Gromov proved that there is a homotopy equivalence between the space of Lagrangian immersions $L \to (W, \omega)$ and a space of bundle maps $TL \to TW$ \cite{gromov:h-princ}.

Rigidity results for higher homotopy groups are less common.  Bourgeois uses the cylindrical contact homology invariant to construct non-trivial examples of elements in $\pi_m$ of the space of contact structures on unit cotangent bundles of negatively curved manifolds \cite{bourgeois:homotopy}.  K\'alm\'an uses the Chekanov-Eliashberg DGA  invariant to construct a non-trivial example in  $\pi_1$ of the space of Legendrian knots in standard contact $\rr^3 $ \cite{kalman:mono1}. K\'alm\'an's example is especially interesting because his loop of Legendrian knots is contractible as a loop of smooth knots.

In this article, we study the space of Legendrian submanifolds in the $1$-jet space $J^1M$ with its canonical contact structure.  The template for finding nontrivial elements in higher homotopy groups is the same as that used in the rigidity results above:  first, to an object $X$ in the space $\mathcal{X}$, associate some (graded) group $H(X)$ which is an invariant of the path component of $X \in \mathcal{X}$.  Next, to an element $\gamma \in \pi_m(\mathcal{X};X)$, associate an element $\Phi(\gamma) \in \End_{1-m}(H_*(X))$, and attempt to prove that this endomorphism is non-trivial.  In contrast to the results above, which use flavors of the holomorphic-curve-based contact homology, we use the generating family homology as our invariant; see \cite{f-r, lisa:links}. 
Because generating homology is a Morse-theory-based homology, 
the advantage of this choice is two-fold: first, our proofs do not have to deal with the technical analysis of a holomorphic curve theory or the complicated combinatorics of the Chekanov-Eliashberg algebra; and second, families of Morse-theory-based homologies have been elegantly packaged in Hutching's language of spectral sequences \cite{hutchings:families}.

Suppose the Legendrian $\leg \subset J^1M$ has  a generating family $f$  with generating family homology $GH_*(f).$
Let $\sleg$ denote the space of Legendrian embeddings in $J^1M.$  The main technical application of the families framework developed in this article is the following:

\begin{thm} \label{thm:morphism}
There exists a morphism from 
 $\pi_m(\sleg(J^1M), \leg)$ to $\End_{1-m}(GH_*(f))$ if $m>1$, or from a subgroup of $\pi_1(\sleg(J^1M), \leg)$ to $\Aut(GH_*(f))$ if $m=1$.
\end{thm}

For the space of Legendrian submanifolds of $\rr^{2n+1}$, with $n>1$, we find that the morphism is nontrivial.

\begin{thm}  \label{thm:non-trivial}
 There exists an infinite family of Legendrian $n$-spheres in $\rr^{2n+1}$
 such that for each sphere $\leg,$ there exists an element $\alpha \in \pi_1(\sleg; \leg)$
 which is contractible as a smooth loop of spheres but is not contractible in the space of Legendrian submanifolds.
\end{thm}

We remark that recently a similar map has been announced by Bourgeois and Br\"onnle. Their map counts certain holomorphic curves, and it is unclear if the two maps are related.

In Section \ref{sec:background}, we review generating families and generating family homology. In Section \ref{sec:spectral-sequence}, we review Hutchings' families framework for families of Morse functions, and adapt it to our set-up of generating families. In Section \ref{sec:algebra}, we prove the main results, finishing by rephrasing Theorem~\ref{thm:morphism} in slightly more general terms. In Section \ref{sec:computations}, we apply the families framework in several ways; for example, to computing generating family holomogy of higher dimensional Legendrians via a bootstrap argument, as well as to showing how the morphism in Theorem~\ref{thm:morphism} factors through front-spinning.

\subsection*{Acknowledgements}

We thank Ryan Budney, Dev Sinha, Octav Cornea, and Michael Hutchings for stimulating conversations about the work in this paper; Ryan Budney was especially helpful in clarifying Proposition~\ref{prop:pi1L-smooth}.  The second author also thanks the Centre de Recherches Math\'ematiques of Montr\'eal for its hospitality during the preparation of this paper.

\section{Background Notions}
\label{sec:background}

In this section, we briefly review the notion of a generating family for a Legendrian submanifold and the (Morse theoretic) generating family homology.

\subsection{Spaces of Legendrian Submanifolds}
\label{ssec:sleg}

Let $J^1M$ denote the $(2n+1)$-dimensional 1-jet space of a $n$-dimensional smooth manifold $M.$ We assume that $M$ is closed, or else diffeomorphic to $\rr^n$ outside of a compact set.  The 1-jet space is equipped with the standard contact structure.  Let $\leg \subset J^1M$ be an $n$-dimensional Legendrian submanifold.  We are interested in the topology of the space of Legendrian submanifolds, which is formed by taking the quotient of the function space of Legendrian embeddings by orientation-preserving self-diffeomorphisms of the domain.  The space of submanifolds inherits the quotient topology from the weak $C^\infty$ topology on the function space, as in \cite{hirsch}.  Let $\sleg(J^1M)$ denote this space of submanifolds, and simply denote by $\sleg^n$ the space of local Legendrian submanifolds, i.e.\ $\sleg(\rr^{2n+1})$.

\subsection{Generating Families for Legendrian Submanifolds}
\label{ssec:gf}

Generating families generalize the fact that the $1$-jet of a function $f: M \to \rr$ is a Legendrian submanifold of $J^1M$.  To see how, begin by considering the trivial fiber bundle $M \times \rr^N$ with coordinates $(x,\e)$. A function $f: M \times \rr^N \to \rr$ is a \dfn{generating family} if $0$ is a regular value of the function $\partial_\e f: M \times \rr^N \to \rr^N$.  Denote by $\mathcal{F}$ the set of all generating families.

A generating family yields a Legendrian submanifold as follows: consider the \dfn{fiber critical set} 
$$\Sigma_f = \left\{(x,\e) \in M \times \rr^N\;:\;\partial_\e f(x,\e) = 0 \right\}.$$
The Legendrian submanifold $\leg_f$ defined by $f$ is then the $1$-jet of $f$ along $\Sigma_f$:
$$\leg_f = \left\{(x, \partial_x f(x,\e), f(x,\e))\;:\; (x,\e) \in \Sigma_f \right\}.$$
Said another way, the Cerf diagram for the family of functions $f_x$ parametrized by $x \in M$ is the front diagram for $\leg_f$.  A given Legendrian submanifold $\leg$ may have many different generating families; call that set $\mathcal{F}_\leg$.  

Let $p: \mathcal{F} \to \sleg(J^1M)$ denote the map that sends a generating family $f$ to the Legendrian submanifold $\leg_f$ that it generates.  A key fact for this paper is:

\begin{thm}[\cite{theret:viterbo}] \label{thm:serre}
  The map $p: \mathcal{F} \to \sleg(J^1M)$ is a Serre fibration.
\end{thm}

\subsection{Generating Family Homology}
\label{ssec:gfh}

Generating families may be used to define a Morse-Floer-type theory for Legendrian submanifolds; see \cite{f-r, lisa:links} as well as \cite{josh-lisa:cap}.  The definition requires the use of Morse theory on non-compact domains, so we restrict our attention to generating families that are either \dfn{linear at infinity} or \dfn{quadratic at infinity}.  The former (resp.\ latter) condition requires the generating family $f$ to agree with a nonzero linear function $A(\e)$ (resp.\ a non-degenerate quadratic function) outside a compact set in $M \times \rr^N$. If $f$ is linear at infinity, then it may be represented as $f = f_0 + A$, where $f_0$ has compact support and $A$ is linear; the \dfn{support} of $f$ is the support of $f_0.$ From here on, we assume that our functions are linear at infinity.

The first step in the definition of generating family homology is to introduce the \dfn{difference function} on the fiber product of the domain of $f$ with itself:
\begin{align*}
  \delta: M \times \rr^N \times \rr^N &\to \rr \\
  (x, \e, \te) &\mapsto f(x,\te) - f(x,\e).
\end{align*}
The critical points of $\delta$ with positive critical values correspond to the Reeb chords of $\leg_f$, and we capture this geometric information with the following definition of \dfn{generating family homology}:
$$GH_k(f) = H_{N+1+k}(\delta^\omega, \delta^\epsilon; \zz/2),$$
where $\omega$ is a number larger than any critical value of $\delta$ and where there are no critical values of $\delta$ in $(0,\epsilon)$.  It is not hard to prove that the groups $GH_k(f)$ are independent of the choices of $\omega$ and $\epsilon$; see \cite[\S3]{josh-lisa:obstr}.  It is worth noting that $0$ is a critical value for $\delta$ whose critical points form a Morse-Bott submanifold diffeomorphic to the Legendrian itself.  Further, if a generating family $f$ is linear-at-infinity, then, after a fiberwise change of coordinates, so is its difference function $\delta$ \cite{f-r}.  We then define the support of $\delta$ to be the support of $\delta_0$ where $\delta = \delta_0 + A$ with $A$ linear.

The basic invariance property of generating family homology is:

\begin{thm}[Traynor \cite{lisa:links}] \label{thm:invariance}
  If $f_s: [0,1] \times M \times \rr^N$ is a $1$-parameter family of generating families that generate a Legendrian isotopy $\leg_s$, then there exists an isomorphism
  $$\Phi_{f_s}: GH_k(f_0) \simeq GH_k(f_1).$$
\end{thm}

Combining this theorem with Theorem~\ref{thm:serre}, we see that the set of all generating family homologies for a Legendrian submanifold $\leg$ is invariant under Legendrian isotopy.


\section{Hutchings' Spectral Sequence}
\label{sec:spectral-sequence}

We  review  Hutchings' construction in \cite{hutchings:families} of a spectral sequence for smooth families of Morse functions  and submanifolds in the context of generating families.  Up to some small modifications, his constructions and results  apply to difference functions of generating families.
We slightly extend the theory developed in   \cite{hutchings:families}  to include parameter spaces that have non-empty boundary.

Our first task is to set notation for the family of difference functions we plan to analyze using Hutchings' scheme.  Fix $0 < \epsilon \ll 1.$ Let $B$ be a finite-dimensional compact manifold, thought of as a parameter space.  
Unlike in  \cite{hutchings:families}, we allow $B$ to have nonempty boundary. Let $\pi: Z \rightarrow B$ be a fiber bundle whose fiber over $b \in B$ is $Z_b = M \times \rr^N \times \rr^N.$ Let $\delta = \{\delta_b : Z_b \rightarrow \rr\}_{b \in B}$ be a family of smooth functions depending smoothly on $b$ that satisfies:
\begin{description}
\item[Genericity] In the complement of a codimension one subvariety of $B$, all critical points of $\delta_b$ with critical value at least $\epsilon$ are non-degenerate, and
\item[Linear-at-Infinity] Outside a compact set $K$ in $M \times \rr^N \times \rr^N$, $\delta_b$ agrees with a fixed nonzero linear function on $\rr^N \times \rr^N$.\end{description}
Let $\nabla: Z \rightarrow B$ be a connection.

To work with Morse homology in this setting, we need to introduce metrics and gradient flows.  We begin by introducing a Morse-Smale pair $(F^B,g^B)$ on the base space $B$, requiring the additional property that  $\delta_b$ is Morse for all $b \in \Crit(F^B).$ 
If $\partial B \ne \emptyset,$ we assume that the component of the negative gradient flow of $F^B$ with respect to $g^B,$ orthogonal to $\partial B,$ is non-zero and points inward.
Let $W$ be the horizontal lift to $Z$ of this negative gradient flow lifted using $\nabla.$
Let $g^Z$ denote a fiberwise metric on $Z$ and let $\xi$ be the negative fiberwise gradient flow of $\delta_b$ with respect to $g^Z.$ 
Finally, we define the vector field 
\begin{equation}
\label{eq:gradient-field}
V = \xi +W,
\end{equation}
which we will use to define differentials in a spectral sequence.  We label this geometric data by the tuple
\[
\Z  := (Z \rightarrow B, \delta, F^B, V).
\]

The zeroes  of $V$ are pairs $p = (b,x)$, where $b\in B$ is a critical point of $F^B$ and $x \in Z_b$ is a critical point of $\delta_b.$
We will consider two complementary gradings: the base grading $i(b; F^B)$ and the fiber grading $i(x;\delta^b).$ The total grading of a zero $p$ of $V$ is $i(p) = i(b; F^B) + i(x;\delta^b).$ 

Hutchings proves in \cite[Proposition 3.4 and p.\ 461]{hutchings:families} that, generically, the stable and unstable manifolds of the zeroes of $V$ intersect transversally  under a slightly different set-up: his fiber $Z_b$ is compact, his base $B$ cannot have boundary, and  0 is not a degenerate critical value.
Even so, since Hutchings' proof works by examining one pair of non-degenerate critical points at a time, his proof still applies to pairs of critical points with positive critical value in our set-up, with the linear at infinity condition taking the place of compactness. We say that $\Z$ is \dfn{admissible} (over $B$) if the choices above are sufficiently generic so that the stable and unstable manifolds of zeroes of $V$ are transverse.

To make the intersections of the stable and unstable manifolds easier to work with, we set some additional notation. Fix zeroes $p$ and $q$ of $V$. Define $\widetilde{\M}(p,q)$ to be the space of negative flowlines $u \in C^\infty(\rr, Z)$ of $V$, i.e.\ smooth maps $u: \rr \to Z$ that satisfy $\frac{d}{dt}u(t) = -V(u(t))$, with the property that $\lim_{t \to -\infty} u(t) = p$ and $\lim_{t \to \infty} u(t) = q$.  We use this set to define the  \dfn{moduli space of flowlines}
\begin{equation*}
  \M(p,q) = \left.\left\{u \in \widetilde{\M}(p,q) \right\} \right/ \sim \\
\end{equation*}
where $u \sim u'$ if $u(t) = u'(t + \tau)$ for some $\tau \in \rr$.  

\begin{prop}
\label{prop:moduli_space}
For a generic choice of $V$, $\M(p,q)$ is a pre-compact manifold of dimension $i(p) - i(q)$.  The boundary of the compactification is given by:
\begin{equation*}
\partial \overline{\M}(p,q)  = \bigsqcup_{r \in \Crit(V)}\M(p,r) \times \M(r,q)
\end{equation*}
\end{prop}

\begin{proof}
This is a rephrasing of the standard argument in Morse homology.  
Note that even though the space $Z$ need not be compact, the linear-at-infinity condition on $\delta$ means that $V$ satisfies the Palais-Smale condition as set down in \cite[\S2.4.2]{schwarz}.

If $\partial B \ne \emptyset,$ we augment the standard argument as follows.
Extend the family to be over a slightly larger open base manifold $B'$ where the fiber $Z_b$
for $b \in B' \setminus B$ is constant in the direction orthogonal to $\partial B.$
Extend the function $F_B$ to $F_{B'}$ such that for a generic metric $g_{B'}$ which extends
$g_B$, the negative gradient flow projected orthogonally to $\partial B$ points towards $\partial B \subset B'$ in any component
of $B' \setminus \partial B.$
Even though $B'$ is not compact, there are no flow lines starting or ending at any critical point 
that flow into $B' \setminus B$; thus, the usual arguments that show that the moduli spaces are manifolds with corners  from Morse theory, applied to $B$, hold.
\end{proof}

Following Hutchings, the data $\mathcal{Z}$ yield a bigraded chain complex 
\begin{equation} \label{eq:chain-complex} \left(C_{l,m} = C_{l,m}(\Z, \epsilon), d = \sum_{n \ge 0} d_n(\Z)\right),
\end{equation} 
where  
the generators are the critical points $(b,x)$ of $V$ 
with $\delta_b(x) > \epsilon$.  The generator $(b,x)$ has  bigrading $(i(b;F^B), i(x;\delta^b))$. 
The differential $d_n: C_{l,m} \rightarrow C_{l-n, m+n-1}$ counts  flow lines of $V$ with coefficients in $\zz/2$.  Specifically, we define:
\begin{equation} \label{eq:filtered-differential} 
d_n((b,x)) := \sum_{(c,y) \in C_{l-n, m+n-1}} \#\M_0((b,x),(c,y)) (c,y). \\
\end{equation}
That the map $d$ is a genuine differential follows from Proposition~\ref{prop:moduli_space}. We filter the complex $C_\nu := \bigoplus_{l+m = \nu} C_{l,m}$ by the first grading,
$F_l C_\nu := \bigoplus_{l' \le l} C_{l', \nu-l'},$ and let $E^*_{*,*} = E^*_{*,*}(\Z, \epsilon)$ be its associated spectral sequence.

The proof of Theorem~\ref{thm:invariance} applies to the current situation, and implies that the fiberwise generating family homologies
$GH_*(f_b)$ can be assembled into a locally constant sheaf, which we denote by $\F_*(\Z).$ 

\begin{thm} \label{thm:MainPrinciple}
Consider the admissible family of  generating families $\Z = (Z \rightarrow B, \delta, F^B, V).$
\begin{description}
\item[$E^2$ term]  The $E^2$ term of the spectral sequence is
  $$E^2_{l,m} = H_l(\F_m(Z)).$$
\item[Homotopy invariance] If $\Z$ is admissible over $B  \times [0,1]$ with the restrictions
$\Z_0 : = \Z|_{\{0\} \times B}$ and $\Z_1 : = \Z|_{\{1\} \times B}$ also admissible, then there is an isomorphism of spectral sequences
$$E^*_{*,*}(\Z_0) = E^*_{*,*}(\Z_1).$$

On the $E^2$ term, this is the isomorphism 
$$H_l(\F_j(\Z_0)) \cong H_l(\F_j(\Z_1))$$ induced by the isomorphism of  local coefficient systems $$\F_j(\Z_0) \cong \F_j(\Z_1)$$
defined by $\Phi$ in Theorem~\ref{thm:invariance}.
\item[Naturality] If $\phi:B' \rightarrow B$ is sufficiently generic so that $\phi^*\Z$ is admissible, then the pushforward in homology
$$\phi_*:H_*(B';\F_*(\phi^*\Z)) \rightarrow 
H_*(B;\F_*(\Z))$$ extends to a morphism of spectral sequences
$$E^*_{*,*}(\phi^*\Z) = E^*_{*,*}(\Z).$$
\item[Trivialty] If $(\delta_b, \xi_b)$ is Morse-Smale for all $b \in B$, then the spectral sequence collapses at the $E^2$ page.
\end{description} 
\end{thm}

\begin{proof}
When $\partial B = \emptyset,$ the properties stated in the theorem follow with little or no modifications from Hutchings' arguments.  In outline, Hutchings first establishes the theorem for spectral sequences defined using singular chains in the base (for \emph{any} base); see Propositions 4.1, 4.3, 4.6 and Remark 1.5 in \cite{hutchings:families}.  Hutchings then extends the isomorphism from singular homology to Morse homology in \cite[Section 2.3]{hutchings:families} to an isomorphism of singular spectral sequences and Morse spectral sequences over closed manifold base spaces in \cite[Proposition 6.1]{hutchings:families}.

When $\partial B \ne \emptyset,$ we need to supplement the arguments connecting singular and Morse homology. The key idea in the argument is that the descending manifold of a critical point is a manifold with corners \cite[Equations (2.6) and (2.7)]{hutchings:families}. That these equations extend to the case of a base manifold with boundary comes from repeating the argument given in the proof of Proposition \ref{prop:moduli_space}.
\end{proof}

\begin{rem}
There are several other properties of Hutchings' spectral sequence that we have not included in the theorem above.  The most interesting is a Poincar\'e duality statement, which holds in our set-up for some cases.  In particular, compare \cite[Lemma 7.1]{josh-lisa:cap} with \cite[Proposition 7.1]{hutchings:families}. A more general duality principle for generating family (co)homology is, however, unclear.
\end{rem}

\section{Algebra of Homotopies}
\label{sec:algebra}

In this section, we use the ideas of Section~\ref{sec:spectral-sequence} to investigate the homotopy groups of the space of Legendrian submanifolds.
In Section \ref{ssec:trace}, we discuss how to interpret a family of $n$-dimensional Legendrians $\Lambda_b \subset J^1M,$ parameterized by the $m$-manifold $B,$ as a single $(m+n)$-dimensional Legendrian $\Lambda.$
We also discuss relationships to the generating family homology.
In Section \ref{ssec:main-result}, where $B = S^m$ is a (based) $m$-sphere, 
we interpret Theorem \ref{thm:MainPrinciple}
as a morphism from the based homotopy groups of the space of Legendrian embeddings $\sleg(J^1 M)$
to the space of endomorphisms of generating family homology.
In Section \ref{ssec:dumbbell}, we study this morphism further to find examples of loops of Legendrian embeddings which are non-contractible as Legendrians submanifolds, but contractible as smooth submanifolds.
In Section \ref{ssec:free-homotopies}, we construct a more general morphism from the free homotopy classes of $\sleg(J^1 M).$ 

\subsection{Tracing Families of Legendrian Submanifolds}
\label{ssec:trace}

We  begin by rephrasing the main concept of Section \ref{sec:spectral-sequence} in the language of Legendrian submanifolds and generating family homology.

Let $\leg_b \subset J^1M$ be a smooth family of $n$-dimensional Legendrian submanifolds parameterized by a compact manifold $B$, possibly with boundary. 
Choosing one generating family $f_b$ for one Legendrian $\leg_b$ determines a family of generating families extending $f_b$ (possibly after stabilization) by the uniqueness
of lifting in the Serre fibration of Theorem \ref{thm:serre}.
Define $f:  B \times M \times \rr^N \rightarrow \rr$ and 
$\delta_b: M \times \rr^N \times \rr^N \rightarrow \rr$ by
\begin{equation}
\label{eq:family-generating-family}
f(b,m,\eta) = f_b(m,\eta),  \quad
\delta_b(m,\eta,\tilde{\eta}) =   f_b(m,\eta) -  f_b(m,\tilde{\eta}).
\end{equation}
Let $\leg \subset J^1(B \times M)$ be the $(n+\dim(B))$-dimensional Legendrian {\bf trace}; that is, the front of $\leg$ over the point $b$ is the front of $\leg_b.$
As in Section \ref{sec:spectral-sequence}, let $F^B: B \rightarrow \rr$ be a generic function on the base, let $V$ be the vector field from equation (\ref{eq:gradient-field}), and let
$\Z = (Z \rightarrow B, \delta = \{\delta_b\}_b, F^B, V).$

\begin{lem}
\label{lem:trace}
The function $f$ is a generating family for $\leg.$ If $F^B$ is a sufficiently $C^2$-small Morse function and $\Z$ is admissible, then
\[
GH_k(f) = \bigoplus_{i+j=k+N+1} E^\infty_{i,j}(\Z).
\]
\end{lem}

\begin{proof}
This result is straightforward after making two observations.  First, in local coordinates, the differential of the fiber derivative of $f$ at $(b,m,\eta)$ contains the differential of the derivative of $f_b$ as a full-rank submatrix.  Thus, $f$ also satisfies the transversality condition for generating families.
Second, the quasi-isomorphism type (which determines its homology) of $CM_*((\delta+F^B)^\omega, (\delta+F^B)^\epsilon)$ is independent
of the choice of generic $F^B$ which makes $\delta$ Morse,
assuming $F^B$ is $C^2$-small, and hence perturbing by $F^B$ does not change the topology of the level $\epsilon$ sublevel set.
\end{proof}

We next consider two examples.
The first will be used in Sections \ref{ssec:main-result} and \ref{ssec:dumbbell}, while the second appears in Section \ref{ssec:free-homotopies}.

\begin{ex}[Based $m$-sphere]
\label{ex:m-sphere}
Let $\Lambda \subset J^1M$ be an $n$-dimensional Legendrian submanifold.
Let $\rho: S^m \rightarrow \sleg(J^1M) $ be a smooth  $S^m$-family of Legendrian submanifolds with the property that for a small contractible neighborhood $U$ of $b \in S^m,$ we have
$\rho(U) = \Lambda.$
Construct a  Morse function $F^{S^m}: S^m \rightarrow \rr$ that has two critical points, a maximum at $a \in U$ and a minimum at $b$.
Assume that $\|F^{S^m}\|_{C^2} < \epsilon$ as in Lemma \ref{lem:trace}.
Let $\Lambda$ be the trace of this $m$-isotopy and define the generating family $f$ for $\leg$ as in Equation~(\ref{eq:family-generating-family}).  If $m=1$, we assume that the application of Theorem~\ref{thm:serre} yields a \emph{loop} of generating families, not just a path.  
Perturb $V$ if necessary so that
\[
\Z = (Z \rightarrow S^m , \delta, F^{S^m}, V)
\]
is an admissible family.
\end{ex}

\begin{ex}[Based homotopy]
\label{ex:based-homotopy}

Let $\Lambda \subset J^1(M)$ be an $n$-dimensional Legendrian submanifold.
Let $\tilde{\rho}: [0,1]^m \rightarrow \sleg(J^1M)$ be a smooth  $[0,1]^m$-family of Legendrian submanifolds
such that $\rho(0, \ldots, 0) = \Lambda.$
Extend $\tilde{\rho}$ to 
${\rho}: I^m:= [-1,1]^m \rightarrow \sleg(J^1M)$ by defining
\[
\rho_{(b_1, \ldots, b_m)} = \tilde{\rho}{(\max(b_1,0), \ldots, \max(b_m,0))}.
 \]
Assume that $\tilde{\rho}|_{\partial [0,1]^{m-1} \times b_m}$ is independent of $b_m.$
Define the Morse function on the base to be:
\begin{equation}
\label{eq:base-function}
F^{I^m}: I^m \rightarrow \rr, \quad F^{I^m}(b_1, \ldots, b_m) = \sigma \sum_{i=1}^m (b_i + 1)^2 (b_i -1)^2,
\end{equation}
where $0 < \sigma \ll \epsilon \ll 1.$
Note that for any metric, the negative gradient of $F^{I^m}$ projects to the outward normal direction on $\partial I^m.$

Let $\Lambda$ be the trace of this $m$-isotopy and define the generating family $f$ and its difference function $\delta$ as in equation (\ref{eq:family-generating-family}).
Perturb $V$ if necessary such that
\[
\Z = (Z \rightarrow I^m , \delta, F^{I^m}, V)
\]
is an admissible family.
\end{ex}

\subsection{From Homotopy Groups of the Space of Legendrians to Generating Family Homology}
\label{ssec:main-result}

We revisit the map $\rho: S^m \rightarrow \sleg(J^1M)$ from Example \ref{ex:m-sphere}, using it to  relate the homotopy groups of $\sleg(J^1M)$ to  morphisms of generating family homology. Specifically, if $f$ is a generating family for $\leg$ and $m>1$, then we will construct a morphism 
\[
	\Psi: \pi_m(\sleg (J^1M); \leg) \to \End_{1-m}(GH_*(f))
\]
If $m=1$, then we restrict the domain of $\Psi$ to the set of homotopy classes of loops in $\sleg (J^1M)$ that lift to loops (not just paths) of generating families; denote by $\pi_1^{gf}(\sleg(J^1M),\leg_0)$ the subgroup associated to those loops.  Note that if a loop in $\sleg (J^1M)$ does not lift to a loop of generating families, then we already know that the loop is non-contractible.

To define the map $\Psi$, we begin by setting notation.  Fix a generating family $f$ for $\leg$ and a small neighborhood $U \subset S^m$ that contains both the  the maximum $a$ and the minimum $b$ of a $C^2$-small function $F^{S^m}$ that has no other critical points. Suppose that $\rho: S^m \rightarrow \sleg(J^1M)$ is a smooth map with the property that $\rho(U) = \leg$.  Construct the generating family $f^\rho$ as in Example~\ref{ex:m-sphere}, recalling that if $m=1$, then we assume that we have a loop of generating families. 

Lemma~\ref{lem:trace} implies that the differential of the generating family chain complex $GC_*(f^\rho)$ in degree $l$ can be written as $d = \sum_{k=0}^{l+1} d_k(\Z)$, as in equations (\ref{eq:chain-complex}) and (\ref{eq:filtered-differential}). For an element $c \in \Crit(F^{S^m}),$ and a generator $(e,p) \in GC_*(f^\rho),$ define  $\langle (e,p),c \rangle$ to be $p \in GC_*(f^\rho_c)$ if $e=c$ and $0$ otherwise.  Extend this pairing bilinearly.
 
Finally, define a map $\psi_\rho: GC_*(f) \rightarrow GC_{*-m+1}(f)$ by:
\begin{equation}
\label{eq:key-map}
\psi_\rho(x) = \begin{cases} \left\langle  d_m(a,x), b \right\rangle + x & m=1, \\
\left\langle  d_m(a,x), b \right\rangle & m>1.\end{cases}
\end{equation}

We can now restate (and prove) Theorem \ref{thm:morphism} is more detail.
\begin{prop}
  \label{prop:pi_k-endo}
	The map $\psi_\rho$ defined above has the following properties:
\begin{enumerate}
\item
The map induces a homomorphism 
$$\Psi_\rho: GH_*(f) \to GH_{*+1-m}(f).$$ 
\item
If $\rho$ and $\rho'$ are homotopic through maps that send $U \subset S^m$ to $\leg_0$, then $\Psi_\rho = \Psi_{\rho'}$.  In particular, given $[\rho] \in \pi_k(\sleg(J^1M), \leg_0)$,  we may refer to the map $\Psi_{[\rho]}$.

\item
The map $\rho \mapsto \Psi_\rho$ induces a morphism from $\pi_m(\sleg(J^1M), \leg_0)$, $m>1$, to $\End_{1-m}(GH_*(f))$ or from $\pi_1^{gf}(\sleg(J^1M), \leg) \to \Aut(GH_*(f))$.  In particular, we have:
\begin{align*}
\Psi_{[\rho][ \sigma]} &= \Psi_{[\rho]}  \Psi_{[\sigma] } & \text{if } m=1, \\
\Psi_{[\rho]+[\sigma]} &= \Psi_{[\rho]}  + \Psi_{[\sigma]} & \text{if } m >1.
\end{align*}
For the $m=1$ case, the equation above implies that $\Psi_{[\rho]}$ is invertible.
\end{enumerate}
\end{prop}

\begin{proof}
 The general principle of this proof is outlined in \cite{hutchings:families}.
  For the convenience of the reader, we present some of the details here when considering generating families.

To prove the first property, note that $d^2(c,x) = 0$ if and only if  
$\langle d^2(c,x), e \rangle = 0$ 
for all $e \in  \Crit(F^{S^m}).$ Since the base function $F^{S^m}$ has critical points of index $0$ and $m$ only, we see that $d_k =0$ unless $k=0,m$.  In particular, for all $x \in \Crit(\delta_a),$ we have:
\begin{eqnarray*}
0 & = & \langle d^2 (a, x), b \rangle \\
& = & \langle (d_0 d_{m} + d_{m} d_0) (a, x), b \rangle.
\end{eqnarray*}

Thus, $\psi_\rho$ is a chain map and induces 
a map
$$\Psi_\rho: GH_*(f^\rho_a) \to GH_{*+1-m}(f^\rho_b).$$

Next, we take two homotopic maps $\rho, \rho': S^m \to \sleg(J^1M)$ with admissible data $\Z$ and $\Z',$ respectively.  
Combining Examples \ref{ex:m-sphere} and  \ref{ex:based-homotopy}, we construct an admissible $\Z[-1,1]$ over $I \times S^m = [-1,1] \times S^m$ such that 
$\Z|_{-1} = \Z = \Z|_0$ and $\Z|_1 = \Z'$.
We then apply Lemma \ref{lem:trace} to define $d = d(\Z[-1,1]).$
There are six critical points of $F^{I \times S^m}$, which we denote by $(n, c)$ where $n \in\{ -1,0,1\}$ and $c \in \{a,b\}.$ Since the base indices lie in the set $\{0,1, m, m+1\}$, the equation $d^2=0$ now implies:
\begin{equation} \label{eq:d2-homotopy}
0 =  \langle (d_0 d_{m+1} + d_{m+1} d_0 + d_1 d_m + d_m d_1) ((0,a), x), (1,b) \rangle.
\end{equation}

Since we are working with a \emph{based} homotopy between $\rho$ and $\rho'$, the map $d_1$ corresponds to the identity map; in particular, we have:
\[
d_1((c,0), x) =((c,1),x) + ((c,-1), x)
\]
for $c \in \{a,b\}$ and  $x \in \Crit(\delta_{(c,0)}) =  \Crit(\delta_{(c,\pm 1)}).$
Thus, Equation~(\ref{eq:d2-homotopy}) indicates that the map $H: GC_*(f^\rho_{(a,0)}) \to GC_{*-m+2}(f^\rho_{(b,1)})$ defined by
\[
H(x) =\left\langle  d_{m+1}((a,0),x), (b,1) \right\rangle,
\]
is a chain homotopy between $\psi_\rho$ and $\psi_{\rho'}$.

The proof of the third statement for $m \geq 2$ essentially appears in \cite[Example 1.9]{hutchings:families}, as Hutchings' proof relies on a  based homotopy similar to the one we  just explicitly constructed.  

For $m=1$, we are unaware how to apply Theorem \ref{thm:MainPrinciple} to prove that $\Psi_{[\rho][\rho']} = \Psi_{[\rho]}  \Psi_{[\rho']}.$ 
Instead, this follows from the traditional ``broken-curves" argument of the more well-studied continuation methods in Morse/Floer theory. 
\end{proof}

\subsection{
A constructive proof of Theorem~\ref{thm:non-trivial}
}
\label{ssec:dumbbell}

In this section, we prove Theorem~\ref{thm:non-trivial}, namely that for every $n>1$, there is an infinite family of Legendrian submanifolds, $\leg^{n,r} \subset \rr^{2n+1}$ parametrized by $r \in \nn$ so that $\pi_1(\sleg^n, \leg^{n,r})$ is non-trivial.  Further, the non-trivial homotopy classes we produce in $\pi_1(\sleg^n, \leg^{n,r})$ are trivial in the smooth category.

We begin by constructing $\leg^{n,r}$.  Consider the Legendrian link in $\rr^3$ whose front projection appears in Figure~\ref{fig:surface-ex-1}.  This link, which is isotopic to the Hopf link, has a generating family $f: \rr \times \rr^N \to \rr$ with the the top strand of the top component generated by critical points of index $r+N$ and the bottom strand of the bottom component generated by critical points of index $N-1.$  Spin the front about its central axis into $\rr^{n+1}$ as in \cite{golovko:higher-spin} to get two Legendrian spheres. Then perform a $0$-surgery along the horizontal dotted $1$-disk in Figure~\ref{fig:surface-ex-1} to get a connected Legendrian sphere $\tilde{\leg}^{n,r}$. That the spinning and surgery constructions yield Legendrian surfaces with generating families is a simple generalization of facts proven in \cite{bst:construct}.

\begin{figure}
  \centerline{\includegraphics{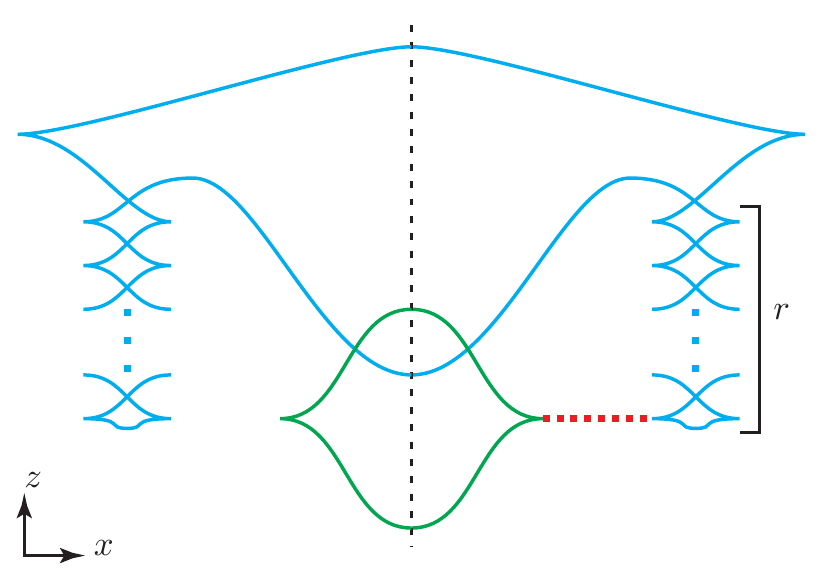}}
  \caption{By spinning this front around the central $z$ axis and then performing a $0$-surgery along the dotted red disk, we obtain the Legendrian surface $\tilde{\leg}^{2,r}.$}
  \label{fig:surface-ex-1}
\end{figure}

To construct $\leg^{n,r}$ itself, we take two copies of $\tilde{\leg}^{n,r}$, positioned sufficiently far apart along the $x_1$ axis so that the pair can be generated by a  single generating family that is equal to a linear function in $\eta$ in a neighborhood of the hyperplane $x_1 = 0$; see \cite[\S3.3]{josh-lisa:cap}.  Finally, perform another $0$-surgery to connect the two copies; once again, the result has a generating family which we will call $f^{n,r}$.

It is important that the three $0$-surgeries performed thus far line up as in Figure~\ref{fig:surface-ex-2}. For $r \geq n+2$, it is straightforward to use the cobordism long exact sequence of \cite{josh-lisa:cap} (see also \cite{bst:construct}) to compute that the generating family homology with respect to the generating family $f^{n,r}$ is:
$$GH_m(f^{n,r}) = \begin{cases} \zz/2 & m = n, \\ \zz/2 \oplus \zz/2 & m=r,1-r \\ 0 & \text{otherwise.} \end{cases}$$
It is easy to see from the computation that the group $GH_r(f^{n,r})$ is generated by two chains $\beta_L$ and $\beta_R$, each of which is arises from a sum of critical points that lie in exactly one of the copies of $\tilde{\leg}^{n,r}$.

\begin{figure}
  \centerline{\includegraphics{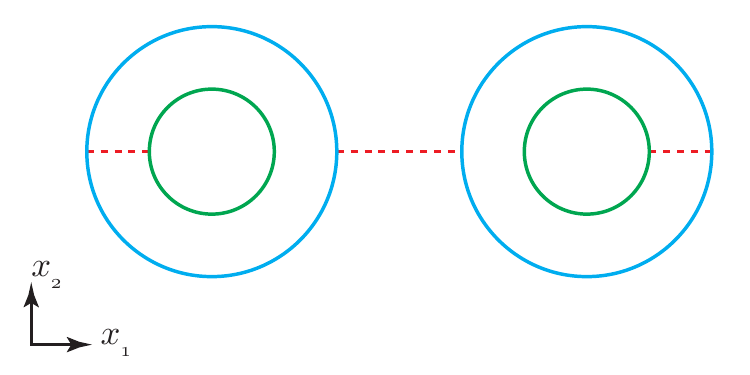}}
  \caption{The three $0$-surgeries in the construction of $\leg^{2,r}$ must line up as in the figure.}
  \label{fig:surface-ex-2}
\end{figure}

With the Legendrian spheres $\leg^{n,r}$ in hand, we proceed to construct a non-contractible loop in $\sleg^n$ based at $\leg^{n,r}$.  The idea is to effect a rotation by $\pi$ in the first two coordinates of the base manifold $\rr^n$, which yields a \emph{loop} in $\sleg^n$ because of the symmetry of $\leg^{n,r}$.  To be more precise, fix $\tau \ll 1$ and choose a smooth function $\sigma: [0,2\pi] \to [0,\pi]$ with the properties that $\sigma$ is non-decreasing, $\sigma^{-1}\{0\} = [0,\tau]$, and $\sigma^{-1}\{\pi\} = [\pi-\tau, 2\pi]$.  Define a path $\rho: [0,2\pi] \to SO(n)$ of rotations of the base $\rr^n$ to be the identity except for the following elements of $SO(2)$ in the upper left corner:
\[
	\begin{bmatrix}
		\cos \sigma(s) & \sin \sigma(s) \\ -\sin \sigma(s) & \cos \sigma(s)
	\end{bmatrix}.
\]
Finally, let $f_s = f^{n,r} \circ \rho(s)$, where we have implicitly extended $\rho$ to be the identity on the fiber component.  The symmetry of the function $f^{n,r}$ implies that this is actually a smooth family of generating families over the base $S^1$ even though $\rho$ does not descend to a smooth function on $S^1$.  In particular, we obtain a smooth loop $\hat{\rho}$ of Legendrian spheres in $\sleg^n$.

To place the construction above in the families context, note that the construction above yields a (trivial) bundle $Z = S^1 \times \rr^n \times \rr^{2N}$ over $S^1$, a fiber-wise difference function $\delta_s$, and a base function $F^B$ as constructed in Section~\ref{ssec:main-result} with maximum at $0$ and minimum at $\pi$.  It remains to specify a vector field $V$.  Choose any metric on the base circle and let $W$ be the lift of $\nabla F^B$ to $Z$ via the trivial connection. Let $\xi_0$ be the fiber-wise gradient of $\delta_0$, and define	
\begin{equation} \label{eq:xi}
	\xi_s(x) = W(s) \rho'(s) + \rho(s) \xi_0(x).
\end{equation}
Finally, as in Section~\ref{sec:spectral-sequence}, we define the vector field $V$ to be $V(x,s) = \xi_s(x) + W(s)$.  Thus, we have all of the data necessary to form a tuple $\Z$ for use in the families construction.

\begin{prop} \label{prop:pikL}
  The loop $\hat{\rho}$ based at $\leg^{n,r}$ is not contractible in $\sleg^n$.
\end{prop}

\begin{proof}
It suffices to show that $\Psi_{\hat{\rho}}$ is not the identity.  

The vector field $V$ constructed above is designed so that a flow line $\gamma(t) = (\gamma_M(t), \gamma_S(t))$ has the following properties:
\begin{enumerate}
\item The component  $\gamma_S(t)$ satisfies the decoupled one-dimensional equation $\gamma_S'(t) = W(\gamma_S(t))$.
\item The component $\gamma_M(t)$ is of the form $\gamma_M(t) = \rho(\gamma_S(t))\zeta(t)$ for some flow line $\zeta(t)$ of the vector field $\xi_0$.  This fact is a straightforward consequence of Equation~(\ref{eq:xi}).
\end{enumerate}
It is then clear that the rigid flow lines that compute the map $\Psi_{\hat{\rho}}$ on $GH_*(f^{n,r})$ send a class of $GH_*(f^{n,r})$ represented by critical points with $x_1 <0$ to the symmetric class represented by critical points with $x_1 > 0$.  By construction, this map is not the identity in degree $r$, and hence the loop $\hat{\rho}$ is not contractible.
\end{proof}

While the loop $\hat{\rho}$ is non-trivial in $\pi_1(\sleg^n,\leg^{n,r})$, it is smoothly trivial. More precisely, we have:

\begin{prop} \label{prop:pi1L-smooth}
  The loop $\hat{\rho}$ is null-homotopic in the space of smooth embedded $n$-spheres in $\rr^{2n+1}$.
\end{prop}

\begin{proof}
  For $n=2$, we exhibit a null-homotopy; by spinning this homotopy, we get a proof for the $n>2$ case.

  The  null-homotopy is constructed in two stages.  First, note that the space of long $2$-knots in $\rr^5$ is connected \cite{budney:embeddings}.  Further, as noted in \cite[Definition 1]{budney:embeddings}, the space of long $2$-knots in $\rr^5$ is homotopy equivalent to the space of embeddings of $D^2$ into $D^5$ that agree with a fixed linear function on the boundary.  Thus,  there is a smooth isotopy of the left lobe of $\leg^{2,r}$ that satisfies the following:
\begin{enumerate}
\item It fixes the attaching region of the $0$-surgery  joining the left to the right lobes;
\item It is supported in the left half-space of $\rr^5$; and 
\item It takes the left lobe to a flying saucer.
\end{enumerate}
Performing this isotopy on the left lobe and its rotation on the right, we obtain a smooth isotopy $H$ that takes $\leg^{2,r}$ down to a flying saucer; note that this isotopy is symmetric about the $z$ axis.  

\begin{figure}
  \centerline{\includegraphics{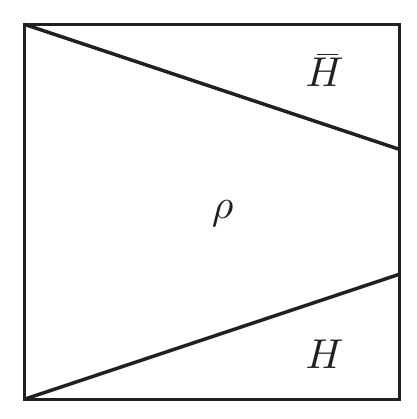}}
  \caption{A schematic picture of the first part of the homotopy between $\rho$ and the constant loop in $\sleg^{2}$.}
  \label{fig:homotopy}
\end{figure}

We are now ready for the first stage of the homotopy $\Theta: [0,2] \to \sleg^2$ that connects $\rho$ to the identity. We work entirely with the front diagram.  At time $t=0$, we simply take $\Theta$ to be $\rho$.  As $t$ increases to $1$, for each fixed $t$, we perform $H(x,3s)$ to gradually transform $\leg^{2,r}$ into the flying saucer over $s \in [0,\frac{t}{3}]$, then rotate the result by $\pi$, and then perform the reverse homotopy $H(x, 3(1-s))$ for $s \in [1-\frac{t}{3},1]$. See Figure~\ref{fig:homotopy} for a schematic picture of this construction.  At $t=1$, the loop $\rho$ has been transformed into a loop that starts by doing $H$ over $[0,\frac{1}{3}]$, then fixes the flying saucer over $[\frac{1}{3},\frac{2}{3}]$, and then undoes $H$ over $[\frac{2}{3},1]$.  This loop is clearly null-homotopic, and we append this null homotopy to the homotopy constructed above.
\end{proof}

Propositions~\ref{prop:pikL} and \ref{prop:pi1L-smooth} together imply Theorem~\ref{thm:non-trivial}.

\begin{rem}
  The proof above shows that the element $\hat{\rho} \in \pi_1(\sleg^n, \leg^{n,r})$ has order at least $2$.  We can modify the construction to produce elements $\hat{\rho}_m \in \pi_1(\sleg^n, \leg^{n,r})$ that have order at least $m$ for any $m>1$.  Instead of connecting two copies of $\tilde{\leg}^{2,r}$ with a $0$-surgery, we begin with a central flying saucer centered on the $z$ axis.  We then take $m$ copies of $\tilde{\leg}^{n,r}$, arrayed as in Figure~\ref{fig:surface-ex-higher-order}, and let $\rho^{m,r}$ be a rotation about the $z$ axis by $\frac{2\pi}{m}$.  The computations of the generating family homology have the same form as those for $\leg^{n,r}$, and a slight generalization of the proof of Proposition~\ref{prop:pikL} shows that all powers $\rho^{m,r}, (\rho^{m,r})^2, \ldots, (\rho^{m,r})^{m-1}$ are nontrivial maps.

\begin{figure}
  \centerline{\includegraphics{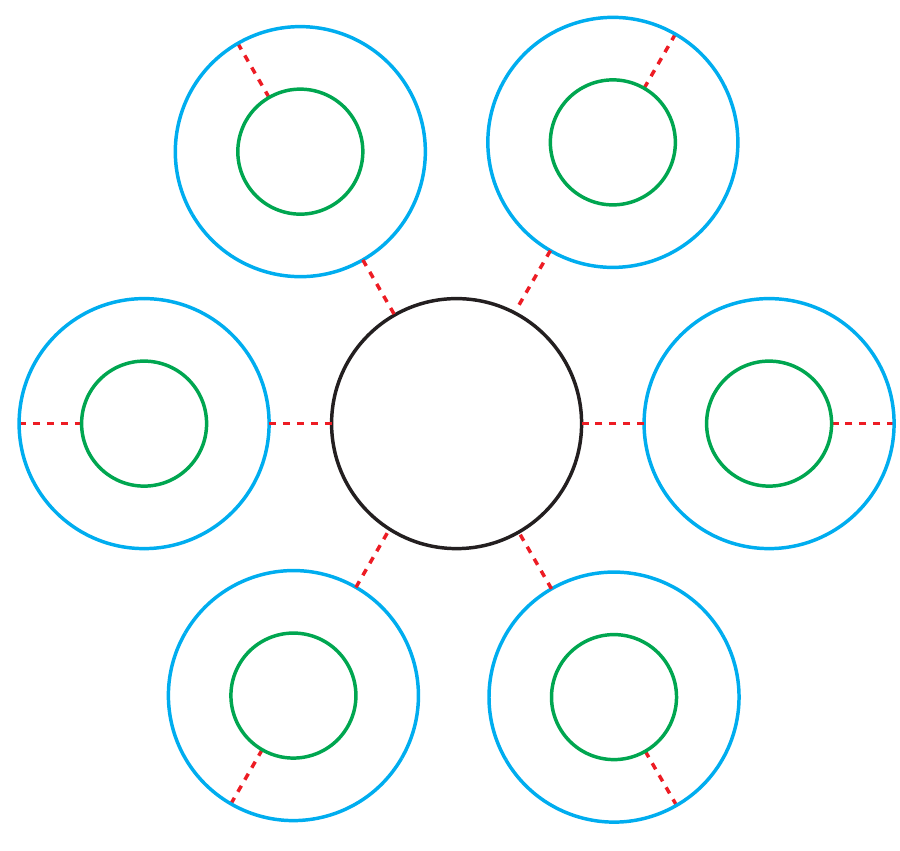}}
  \caption{The fundamental group of $\sleg^2$ based at this surface has an element of order at least 6.}
  \label{fig:surface-ex-higher-order}
\end{figure}

In fact, the argument above shows that for any subgroup $G < SO(n)$ that acts transitively and without fixed points on a set $S \subset S^{n-1}$, there exists an $n$-dimensional Legendrian submanifold $\leg_G \subset \rr^{2n+1}$ and an injection $G \hookrightarrow \pi_1(\sleg^n, \leg_G)$.
\end{rem}

\subsection{Free homotopies}
\label{ssec:free-homotopies}

One can also consider relative versions of the discussion of the map $\Psi$: instead of $m$-spheres of Legendrians up to basepoint-preserving homotopy, consider $m$-cubes of Legendrians up to homotopy relative to their boundary.
One way to algebraically package this, before passing to homology, is as a \dfn{fundamental $\infty$-groupoid}, which we sketch below. This groupoid is an example of a so-called $(\infty,0)$-category.  Essentially, an $(\infty,0)$-category
is a category with objects, 1-morphisms between objects,
2-morphisms between 1-morphisms, etc.
The ``$(\cdot, 0)$''-label 
indicates that all $k$-morphisms for $ k>0$ have homotopy inverses.
The  $``(\infty,\cdot)"$-label indicates that
operations  and relations, such as the composition of two composable 1-morphisms and associativity of composition, only hold up to ``homotopy."  
For a rigorous definition of an $(\infty,0)$-category
in terms of Kan complexes and simplicial sets, see \cite[Remark 1.1.2.3 and Example 1.1.2.5]{lu:higher-topos}

\begin{ex}
\label{ex:fundamental-goupoid}
As mentioned, an example of an $(\infty,0)$-category is $\pi_{\le \infty}(X),$ the {fundamental $\infty$-groupoid} of a topological space $X.$ 
The objects of $\pi_{\le \infty}(X)$ are the points in $X.$ 
The 1-morphisms $Mor_1(x,y)$ are the (possibly empty set of) paths from $x$ to $y.$ 
Composition of composable 1-morphisms is concatenation of paths. Note that we are unconcerned with how to parameterize the composite path since all choices are homotopic. This leads to 
the 2-morphisms $Mor_2(\alpha, \beta)$  between paths $\alpha, \beta$ which start and end at $x,y \in X:$ they are the based homotopies connecting $\alpha, \beta.$
Note that all $(\ge 1)$-morphisms have homotopy inverses.
\end{ex}

\begin{ex}
\label{ex:GH-infinity-category} We define another $(\infty,0)$-category, $\mathcal{GH}(\sleg_n(J^1M)),$ based on the generating family chain complexes of points in $\sleg_n(J^1M).$
The objects are $GC_*(\Z):= GC_*( f)$ with differentials $d = d(\Z).$
Note if $GC_*(\Z) = GC_*(\Z'),$ but the Legendrians $f$ and $f'$ generate are not the same,
the chain complexes are considered the same object in this category.
Given a Legendrian isotopy $\leg_b,$ $ -1 \le b \le 1$ which is constant for $-1 \le b \le 0,$ let $\Z$ be the admissible family associated to the trace $\leg.$ 
(See Section \ref{ssec:main-result}.)
Define a 1-morphisms 
\[
\alpha = \alpha(\Z) \in 
Mor_1(GC_*( f_{-1}), GC_*(f_{1})),
\quad 
\alpha(x) :=\left\langle  d_{1}({\bf{0}},x), {\bf{1}} \right\rangle.
\]
(using the notation of the proof of Proposition \ref{prop:pi_k-endo}).
Note that when defining $Mor_1(GC_*(\Z), GC_*(\Z'))$, we are considering all families  $\Z[-1,1]$
between {\em{all}} pairs
 $\Z$ and $\Z'$ (as in the proof of Proposition \ref{prop:pi_k-endo}) such that
$GC_*( \Z) = GC_*$ and $GC_*(\Z') = GC_*'$.
We continue in this manner, defining the 2-morphisms with the $d_2$-map, et cetera.
\end{ex}

\begin{prop}
\label{prop:fundamental-groupoid}
There is a functor from $\pi_{\le \infty}(\sleg_n(J^1M))$ to $\mathcal{GH}(\sleg_n(J^1M)).$
\end{prop}

\begin{proof}
The proposition follows from almost identical arguments to the proof of Proposition \ref{prop:pi_k-endo}.
\end{proof}

\section{Further Applications}
\label{sec:computations}

In this section, we examine several explicit constructions of families of Legendrian submanifolds with generating families, teasing out the implications of the families machinery of Section~\ref{sec:spectral-sequence} for each construction.

\subsection{Product Families}
\label{ssec:prod-fam}

Suppose that $\leg \subset J^1M$ is a Legendrian submanifold with generating family $f$.  Given a closed manifold $B$, we may form the \dfn{product family} $\leg \times B \subset J^1(M \times B)$ simply by taking the generating family $f^B$ with fiber $f^B_b = f$.  This construction, together with a choice of a $C^2$-small Morse function $F^B$ on $B$ and a metric $g$ on $M \times \rr^N$, induces a family $(Z \rightarrow B, \delta, F^B, V)$.  We may then use Theorem~\ref{thm:MainPrinciple} to compute the generating family homology of the constant family $f^B$ on the total space $\leg \times B$ using a K\"unneth-type formula.

\begin{prop} \label{prop:kunneth}
The generating family homology of the total space of a product family may be computed by:
$$GH_k(f^B) = \bigoplus_{l=0}^{\dim B} GH_l(f) \otimes H_{k-l}(B).$$
\end{prop}

\begin{proof}
The $E^2$ property of Theorem~\ref{thm:MainPrinciple}  implies that 
\[
	E^2_{i,j} = H_i(B; GH_j(f)).
\]
The triviality property of Theorem~\ref{thm:MainPrinciple} implies that the  spectral sequence $E^*_{*,*}$ collapses at the $E^2$ page, and we recover the generating family homology of the family $f^B$ as in the statement of the theorem.
\end{proof}

\begin{cor}
  Suppose that the  Legendrian submanifolds $\leg_1, \leg_2 \subset J^1M$ have different sets of generating family homologies. If $B$ is any closed manifold, then $\leg_1 \times B$ and $\leg_2 \times B$ are not Legendrian isotopic in $J^1(M \times B).$
\end{cor}

While the result of this corollary has been obtained when $M = \rr^n$ and $B$ is the $k$-torus \cite{high-d-duality}, this is a new result for all other cases.

To see an application of the corollary, one may take any pair of twist knots in $J^1\rr$ that Chekanov distinguished using linearized Legendrian contact homology \cite{chv}.  In this case, since the twist knots have only one possible linearized contact homology group, it is easy to use Fuchs and Rutherford's results in \cite{f-r} to show that Chekanov's twist knots have different generating family homology.

\begin{rem}
	The product families construction is a special case of Lambert-Cole's Legendrian product construction \cite{lambert-cole:products}.  The $1$-jet of $F^B$ in $J^1B$ is a Legendrian $\leg_B$ isotopic to the zero section, and the product above is then Lambert-Cole's Legendrian product $\leg \times \leg_B$.
\end{rem}


\subsection{Front Spinning}
\label{ssec:spin}

In the next few subsections, we bring the front spinning constructions of \cite{ees:high-d-geometry, golovko:higher-spin}, their adaptation to generating families \cite{bst:construct}, and their generalization to twist spinning \cite{bst:construct} into the families context.  

For the simplest version of this construction, suppose that a Legendrian submanifold $\leg \subset \rr^{2n+1}$ is contained in the half-space $H$ defined by  $x_n > 1$.  
This can always be achieved via a translation in the $x_n$ direction, which is a Legendrian isotopy.
Suppose further that $\leg$ has a linear-at-infinity generating family $f$ whose support (Section \ref{ssec:gfh})  also lies in the half-space $H$. 
As alluded to in Section \ref{ssec:gfh},
we may also assume that $\delta$ is linear-at-infinity and has support in the half-space $H$ --- in fact, we assume that the support lies in the set defined by $x_n>1$; see \cite{josh-lisa:obstr}.  

We define a new generating family for an $(n+m)$-dimensional Legendrian in $\rr^{2(n+m)+1}$ as follows:  let $(\rho, \boldsymbol{\theta})$ denote generalized spherical coordinates on $\rr^{m+1}$; hence, we may represent a point in $\rr^{n+m} = \rr^{n-1} \times \rr^{m+1}$ by 
$(x_1, \ldots, x_n, \rho, \boldsymbol{\theta})$.  Define the generating family for the spun Legendrian by:
\begin{equation} \label{eq:spin}
f_{\Sigma,m}(x_1, \ldots, x_{n-1}, \rho, \boldsymbol{\theta}, \eta) = f(x_1, \ldots, x_{n-1},\rho, \eta). 
\end{equation}

It is straightforward to check, as noted in \cite{bst:construct}, that $f_\Sigma$ is still a generating family.  We call the new Legendrian the \dfn{$m$-spinning} of $\leg$ and denote it
by $\Sigma^m \leg;$ it clearly has the diffeomorphism type of $\leg \times S^m$.

A small generalization of the proof of Proposition~\ref{prop:kunneth} yields:

\begin{prop} \label{prop:spin}
The generating family homology of the $m$-spun generating family $f_{\Sigma,m}$ may be computed as:
$$GH_k(f_{\Sigma,m}) = GH_k(f) \oplus GH_{k-m}(f).$$
\end{prop}

\begin{proof}
	The proof is structured around a relative Mayer-Vietoris argument in the domain of $\delta_{\Sigma,m}$, where we take the set $A^h$ to consist of points $(x,\rho, \boldsymbol{\theta},\eta) \in \rr^{n+m} \times \rr^{2N}$ with $\rho < 1$ and $\delta < h$ and the set $B^h$ to consist of points with $\rho>\frac{1}{2}$ and $\delta < h$. Since $\delta$ is a linear function for $\rho<1$, we see that the pairs $(A^\omega, A^\epsilon)$ and $(A^\omega \cap B^\omega, A^\epsilon \cap B^\epsilon)$ are both acyclic.  Thus, a Mayer-Vietoris argument shows that $GH_*(f_{\Sigma,m})$ is isomorphic to $H_{*+N+1}(B^\omega,B^\epsilon)$, which, by examination of Equation~\ref{eq:spin}, is precisely	the generating family homology of the product family $\leg \times S^m$ constructed in the previous section. 	
\end{proof}

We conclude, as in the previous section, that if two Legendrians may be distinguished by their generating family homology, then their $m$-spins are so distinguished as well; see \cite[Section 5]{high-d-duality} for a comparable computation for Legendrian Contact Homology when $m=1$.


\subsection{Twist Spinning}

To generalize the spinning construction of Section \ref{ssec:spin}, consider a representative $\alpha$ of an element in $\pi_m(\mathcal{L}^n; \leg)$.  Suppose that $\leg$ has a generating family $f$, and let $f_{\boldsymbol{\theta}}$ denote the lift of $\alpha$ to the set of generating families for $\leg_{\mathbf{\theta}}$ starting at $f$.  If $m=1$, we must explicitly assume that the lifting procedure yields a loop, not just a path, of generating families.  As a common generalization of \cite{bst:construct} and \cite{golovko:higher-spin}, and in parallel to \cite{ek:isotopy-tori} for $m=1$, we define a generating family for the \dfn{twist-spun} Legendrian $(n+m)$-submanifold $\leg_\alpha$ by:
\begin{equation}
f_\alpha(x_1, \ldots, x_{n-1}, \rho, \boldsymbol{\theta}, \eta) = f_{\boldsymbol{\theta}}(x_1, \ldots, x_{n-1},\rho, \eta).
\end{equation}
Front spinning is obviously a special case of twist spinning:  simply twist-spin the constant isotopy.

To compute $GH_*(f_\alpha)$, we return to the setup in Example~\ref{ex:m-sphere}, where the base function $F: S^m \to \rr$ has a maximum at $a \in S^m$, a minimum at $b \in S^m$, and no other critical points.  Theorem~\ref{thm:MainPrinciple} implies that the $E^2$ term of the families spectral sequence for the family $f_{\boldsymbol{\theta}}$ is $GH_*(f) \oplus GH_*(f)[1-m]$ with the differential defined as follows. If $x$ is a generator of $GH_*(f)$, then in the notation of Sections~\ref{sec:spectral-sequence} and \ref{sec:algebra}, the generators of the $E^2$ term are of the form $(a,x)$ and $(b,x)$.  The definition of the map $\Psi$ then implies that the differential is:
\begin{align*}
	d(a,x) &= \begin{cases} 
(b,\Psi_{[\alpha]}(x)+x) & m=1 \\ 
(b, \Psi_{[\alpha]}(x)) & m > 1
 \end{cases} \\
	d(b,x) &= 0.
\end{align*}

\begin{prop}
The generating family homology $GH_*(f_{\alpha})$ is independent of the choice of representative of $\alpha$ and may be computed from the chain complex
$\left(GH_*(f) \oplus GH_*(f)[1-m], d\right)$ described above.
\end{prop}

\begin{proof}
	The proof is parallel to that of Proposition~\ref{prop:spin}, above, with the construction of $\Psi$ in Equation~(\ref{eq:key-map}) and Proposition~\ref{prop:pi_k-endo} taking the place of Proposition~\ref{prop:kunneth}.
\end{proof}

The theorem above can give us information in two ways:  first, it allows us to use distinct elements of $\pi_m(\leg^n;\leg_0)$ to produce pairs of distinct $(n+m)$-dimensional Legendrian submanifolds.  For example, twist-spinning the Legendrian $\leg$ constructed in Section~\ref{ssec:dumbbell} by the non-trivial element in $\pi_1(\leg^n, \leg)$ yields a Legendrian $(n+1)$-submanifold distinct from the ordinary spin of $\leg$. 

The theorem above also provides a potential mechanism to distinguish elements of $\pi_m(\mathcal{L}^n)$: if the twist-spins of two loops of Legendrian with a common base point have different generating family homology, then the difference must have arisen from the $\Psi$ maps.  Thus, if one can compute the generating family homology by some other means --- surgery \cite{josh-lisa:obstr} or a generating family version of the Mayer-Vietoris sequence of \cite{hs:bordered-lch}, for example --- then one has a chance of finding new examples of non-trivial elements of $\pi_m(\mathcal{L}^n)$ without directly computing the $\Psi$ maps directly.  Unfortunately, as of this writing, we know of no implementations of this technique.


\subsection{Factoring $\Psi$ Through Spinning}
\label{ssec:factoring}

In this section, we study the relationship between the morphism $\Psi$ from homotopy groups of spaces of Legendrians and the $1$-spinning construction.  
Unlike in Section \ref{ssec:spin}, we need the analyze the chain complex more closely, but
along the way, we reprove Proposition \ref{prop:spin} in the $1$-spun case.

First we adapt a technique useful for gradient flow trees and holomorphic disks in Legendrian Contact Homology \cite{ees:high-d-geometry, hs:bordered-lch} to generating family homology. 
We state the lemma more generally than is needed in this article for possible future applications.
Let $g$ be a metric on $M \times \rr^N \times \rr^N,$ $S \subset M$ be a submanifold, and
 $N_\epsilon(S) \subset M$ be the $\epsilon$-neighborhood of $S.$  
 Let $\delta$ be the difference function of a generating family $f: M \times \rr^N \to \rr.$
Let $V$ be a (negative) gradient-like vector field for $\delta$ used to define the differential in $GC(f).$
Assume the support of $V$ agrees with the support of $\delta.$

\begin{lem}\label{lem:pinching}
For all sufficiently small $\epsilon>0,$ and for all $(x, \eta, \tilde{\eta})$ such that $x \in \partial N_\epsilon(S)$ and $\delta(x, \eta, \tilde{\eta})>0,$ assume one of the following holds:
either the component of $V$ normal to $\partial N_\epsilon(S)$ is non-vanishing and points inwards; or, $(x, \eta, \tilde{\eta})$ is not in the support of $\delta.$   
Fix points $p,q \in M \times \rr^N \times \rr^N$  with $\delta(p) > \delta(q) > 0$ and negative gradient-like flow line $\gamma$ of $\delta$ connecting them.
  \begin{enumerate}
  \item When $S$ is a hypersurface,
  $\gamma$ does not cross $S \times \rr^N \times \rr^N.$ 
  \item If both $p$ and $q$ lie in $S \times \rr^N \times \rr^N$, then $\gamma$ sits entirely in $S\times \rr^N \times \rr^N.$
  \item
  If $f_S$ is the restriction of $f$ to $S \times \rr^N$, then $GC(f_S)$ is naturally a subcomplex of $GC(f).$
  \end{enumerate}
 If we replace  ``inwards'' with ``outwards'' in the first assumption,
then the first and second statements above still hold.
\end{lem}

\begin{proof}
Note that if $\gamma$ exits the support of $V,$ it then stays within a single fiber $\{x\} \times \rr^N \times \rr^N$.
Thus, for the first statement, it suffices to observe that the hypotheses imply that $V$ is everywhere tangent to $S \times \rr^{N} \times \rr^N$. 

For the second statement, since the normal component of $V$ always points into $T(S \times \rr^{N} \times \rr^N)$ at $p$, or vanishes, even if $p$ is a critical point of $\delta$, the flow line cannot leave any $\epsilon$ neighborhood of $S \times \rr^N \times \rr^N$. Thus, the first observation implies that $\gamma$ lies entirely in $S \times \rr^{N} \times \rr^N.$  A similar proof, based at $q$, holds if we replace the ``inwards'' assumption by ``outwards''.

For the third statement, note that the vanishing normal component of $V$ along $S \times \rr^{N} \times \rr^N$  implies that there is a one-to-one correspondence between the critical points of $\delta$ and those of $\delta_S$.
The equality of differentials then follows from the argument for the second statement which prevents a flow line from leaving $S \times \rr^{N} \times \rr^N.$
\end{proof}

We now study the interaction of spinning and Proposition \ref{prop:pi_k-endo}.
Fix a Legendrian submanifold $\leg \subset \{\rho:= x_n > 1\} \subset J^1\rr^n$ with generating family $f$
whose support lies in $\{\rho > 1/2\} \subset \rr^{n} \times \rr^N.$
A 1-spin produces a Legendrian $\Sigma^1 {\leg} \subset J^1\rr^{n+1}$ with 
generating family $f_{\Sigma,1}$ as in equation (\ref{eq:spin}).
Choose a smooth monotonic function $\lambda(\rho)$ such that $\lambda|[0,1/2] =0$ and
$\lambda|[1,\infty) = 1.$
Fix a small $\epsilon>0$, and let $V$ be the gradient vector field of the difference function
with a $C^2$-small perturbation:
\[
f_{\Sigma,1}(x_1, \ldots, x_n, \rho, \theta, \eta) - 
f_{\Sigma,1}(x_1, \ldots, x_n, \rho, \theta, \tilde{\eta})
+ \epsilon \lambda(\rho) \sin(\theta).
\]

All critical points of the gradient-like vector field $V$ have  
coordinates $\rho > 1$ and $\theta = -\pi/2$ or $\pi/2,$ 
which we distinguish by labeling as $c[-]$ and $c[+],$ respectively, 
where $c$ is a critical point of the difference function of $f$.
This induces a decomposition of the differential $d_{\Sigma,1}$ of $GC(f_{\Sigma,1}) = GC[-] \oplus GC[+]$:
\[
d_{\Sigma,1} = \begin{bmatrix} d_{--} & d_{-+} \\ d_{+-} & d_{++} \end{bmatrix}
\]
We first prove a lemma which implies Proposition \ref{prop:spin} for the $1$-spin case.

\begin{lem}
\label{lem:matrix-factorizing} For all critical points $b,c$ of the difference function of $f$, we have:
\begin{align*}
d_{-+}c[-] &=0, \\
d_{+-}c[+] &= 0,\\
 \langle d_{--} c[-], b[-] \rangle &= \langle d c, b \rangle  =
 \langle d_{++} c[+], b[+] \rangle,
 \end{align*}
 where $d$ is the differential of $GC(f).$
\end{lem}

\begin{proof}
By the symmetry of $V$ under the reflection through the $x_1 \cdots x_{n-1} z$ plane,
any elements in any
rigid moduli space $\M_0(c[+], b[-])$ 
appear in pairs; thus, $ d_{+-} = 0.$

Let $S\subset \rr^{n-1} \times \rr^2$ be the open hypersurface satisfying $\theta = -\pi/2$ and  $\rho > 1/2.$
We see that the hypotheses (with ``inward'' specification) of Lemma \ref{lem:pinching} hold;
 therefore, the third statement of the lemma implies:
 \[
 d_{-+} = 0, \quad \mbox{and} \quad
 \langle d_{--} c[-], b[-] \rangle =   \langle d c, b \rangle.
 \]

Finally, let  $S'\subset \rr^{n-1} \times \rr^2$ be the hypersurface defined by $\theta = \pi/2$ and $\rho > 1/2.$
The identity $\langle d_{++} c[+], b[+] \rangle =  \langle d c, b \rangle$ now follows from
the second statement of Lemma \ref{lem:pinching} (with the ``outward'' hypothesis).
\end{proof}

\begin{prop}
\label{prop:map-factorizing}
Let $\Psi$ be the map from Proposition \ref{prop:pi_k-endo}.
Let $Pr_\pm$ be the projection map defined on generators as
$$GH(f_{\Sigma,1}) \rightarrow GH(f), \quad
 c[\pm] \rightarrow c, \,\, c[\mp] \rightarrow 0.$$
Define the map $i:\pi_m(\sleg(J^1 \rr^n); \leg) \rightarrow \pi_m(\sleg(J^1\rr^{n+1});  \Sigma^1 {\leg} )$ induced by 1-spinning $S^m$ families of Legendrians.
Then $i$ is well-defined, and 
 $\Psi$ factors through 1-spinning, i.e.\ the following diagram commutes:
\[
\xymatrix{
	\pi_m(\sleg(J^1 \rr^n); \leg) \ar[d]^i \ar[r]^\Psi & \End_{1-m}(GH_*(f)) 
	 \\
	\pi_m(\sleg(J^1\rr^{n+1});  \Sigma^1 {\leg} ) \ar[r]^\Psi & \End_{1-m}(GH_*({f}_{\Sigma,1}))\ar[u]^{Pr_\pm}.
}
\]
\end{prop}

\begin{proof}
First note that $i$ is well-defined, since the 1-spin of a homotopy of two Legendrian $S^m$-families is a homotopy of two 1-spun Legendrian $S^m$-families.

Let $d_m$ be the chain map which induces the upper arrow $\Psi$  in the proposition,
and $d^{\Sigma,1}_m$ be the chain map which induces the lower $\Psi,$ both as in
equation (\ref{eq:key-map}).
Using the notation of Lemma \ref{lem:matrix-factorizing}, it suffices to show that: 
\begin{equation}
\label{eq:d_m-on-slice}
\langle d^{\Sigma,1}_m c[-], b[-]\rangle = \langle d_m c, b \rangle = \langle d^{\Sigma,1}_m c[+], b[+]\rangle.
\end{equation}
We prove the first equality, as the second  one follows from identical reasoning.

Let ${\leg}(t),$ $t \in S^m,$
represent an arbitrary element in $\pi_m(\sleg^n; \leg))$ and 
$\Sigma^1 {\leg}(t),$ be its front-spun counterpart.
Recall the $S^m$-family is described in Example \ref{ex:m-sphere}.
For $t \in S^m,$ choose (smoothly in $t$) the half-hyperplane $S(t)$ 
from the proof of Lemma \ref{lem:matrix-factorizing} 
(rotated according to $t$)  which
``cuts out" a copy of $\leg(t)$ from $\Sigma^1 {\leg}(t).$
This defines a hypersurface $S$ in $S^m \times \rr^{n+1}.$
Like in the proof of Lemma \ref{lem:matrix-factorizing}, we see that the hypotheses of Lemma \ref{lem:pinching} are satisfied.
Equation \ref{eq:d_m-on-slice} follows from the second statement of Lemma \ref{lem:pinching}. 
\end{proof}


\bibliographystyle{amsplain} 
\bibliography{main}

\end{document}